\documentclass[12pt, reqno]{amsart}
\usepackage{amsmath, amsthm, amscd, amsfonts, amssymb}
\usepackage{bm}

\newtheorem{theorem}{Theorem}[section]

\newtheorem{proposition}[theorem]{Proposition}

\theoremstyle{definition}

\newtheorem{example}[theorem]{Example}

\numberwithin{equation}{section}

\newcommand{\vp}{\varphi}

\newcommand{\clh}{\mathcal{H}}
\newcommand{\clk}{\mathcal{K}}

\newcommand{\clm}{\mathcal{M}}

\newcommand{\D}{\mathbb{D}}

\newcommand{\T}{\mathbb{T}}
\newcommand{\Z}{\mathbb{Z}}

\newcommand{\raro}{\rightarrow}

\newcommand{\br}{B_{\sigma,e^{i\theta} }}
\newcommand{\brst}{B_{\sigma,e^{i\theta} }^*}
\newcommand{\C}{\mathbb{C}}

\begin{document}


\setcounter{page}{1}


\title[Invariant subspaces and $C_{00}$-property]{Invariant subspaces and the $C_{00}$-property of Brownian shifts}

\author[Das]{Nilanjan Das}
\address{Indian Statistical Institute, Statistics and Mathematics Unit, 8th Mile, Mysore Road, Bangalore, 560059,
India}
\email{nilanjand7@gmail.com}

\author[Das]{Soma Das}
\address{Indian Statistical Institute, Statistics and Mathematics Unit, 8th Mile, Mysore Road, Bangalore, 560059, India}
\email{dsoma994@gmail.com}

\author[Sarkar]{Jaydeb Sarkar}
\address{Indian Statistical Institute, Statistics and Mathematics Unit, 8th Mile, Mysore Road, Bangalore, 560059,
India}
\email{jay@isibang.ac.in, jaydeb@gmail.com}




\begin{abstract}
We consider the restriction of Brownian shifts to their invariant subspaces and classify when they are unitarily equivalent. Additionally, we prove an asymptotic property stating that normalized Brownian shifts belong to the classical $C_{00}$-class.
\end{abstract}

\maketitle

\tableofcontents

\section{Introduction}\label{sec intro}

Let $\sigma > 0$ and $\theta\in [0, 2\pi)$. The \textit{Brownian shift} of covariance $\sigma$ and angle $\theta$ is the bounded linear operator $\br:H^2\oplus \mathbb{C}\to H^2\oplus\mathbb{C}$ defined by
\[
\br = \begin{bmatrix}
S & \sigma (1\otimes 1)\\
0 & e^{i\theta}
\end{bmatrix},
\]
where $S: H^2 \raro H^2$ is the shift operator on $H^2$ and $H^2$ denotes the Hilbert space of all analytic functions in the open unit disc $\D = \{z \in \mathbb{C}: |z| < 1\}$ whose power series coefficients are square summable. Equivalently, an analytic function $f$ on $\D$ is in $H^2$ if and only if \cite{Gar}
\[
\|f\|_2 := \left(\sup_{r \in (0,1)}\left[\frac{1}{2 \pi}  \int_{0}^{2 \pi} |f(r e^{i \theta})|^2 d\theta\right] \right)^{\frac{1}{2}} < \infty.
\]
Recall that $S$ is the multiplication operator by the coordinate function $z$, that is,
\[
S f = z f,
\]
for all $f \in H^2$. Moreover, the operator $1 \otimes 1 : \mathbb{C} \raro H^2$ is defined by $((1\otimes 1)\alpha)(z) = \alpha$ for all $\alpha \in \mathbb{C}$ and $z \in \D$. Brownian shifts were introduced by Agler and Stankus in the context of $m$-isometries \cite[Definition 5.5]{Agler-Stankus}. These operators are related to the time-shift operators associated with Brownian motion processes.

Determining the lattice of closed subspaces that are invariant under a given bounded linear operator is always an interesting problem. When the underlying operator is simple—particularly for naturally occurring operators—the problem becomes even more intriguing. In the case of the Brownian shifts, Agler and Stankus resolved the invariant subspace problem \cite{Agler-Stankus}. They proved that for a nonzero closed subspace $\clm$ of $H^2 \oplus \mathbb{C}$, $\clm$ is invariant under $\br$ if and only if it admits one of the following representations:
\[
\clm = \psi H^2 \oplus \{0\},
\]
for some inner function $\psi \in H^\infty$, or
\[
\clm = \mathbb{C}\begin{bmatrix} g\\ 1\end{bmatrix} \oplus \left(\varphi H^2\oplus\{0\}\right),
\]
where $\vp \in H^\infty$ is an inner function with the condition that $\vp(e^{i\theta})$ exists, and
\[
g(z) = \sigma \left(\frac{\overline{\vp(e^{i\theta})} \varphi - 1}{z-e^{i\theta}}\right) \qquad (z \in \D).
\]
Following Agler and Stankus, we call the former invariant subspace \textit{Type I} and the latter one \textit{Type II}. We also call them the \textit{canonical representations} of the invariant subspaces of $\br$. Here, $H^\infty$ denotes the space of all bounded analytic functions on $\D$. Recall that a function $g \in H^\infty$ is called \textit{inner} if $|g(z)|=1$ almost everywhere on $\mathbb{T}$, in the sense of radial limits \cite{Gar}.

Therefore, the lattice of invariant subspaces of a given Brownian shift operator is parameterized by inner functions in the Type I case, and in the Type II case, by inner functions satisfying a constraint on the existence of a boundary value at $\mathbb{T}$ corresponding to the given angle of the Brownian shift.

In this paper, we take the Agler-Stankus invariant subspace result to the next natural step. Specifically, given a pair of closed subspaces $\clm_1$ and $\clm_2$ of $H^2 \oplus \mathbb{C}$ that are invariant under the Brownian shifts $B_{\sigma_1, e^{i\theta_1}}$ and $B_{\sigma_2, e^{i\theta_2}}$, respectively, we consider the restriction operators $B_{\sigma_1, e^{i\theta_1}}|_{\clm_1}$ and $B_{\sigma_2, e^{i\theta_2}}|_{\clm_2}$ on $\clm_1$ and $\clm_2$, respectively, and determine when they are unitarily equivalent. One of the main results of this paper states the following (see Theorem \ref{thm: unit equiv}): There exists a unitary operator $U : \clm_1 \raro \clm_2$ such that
\[
U B_{\sigma_1, e^{i\theta_1}}|_{\clm_1} = B_{\sigma_2, e^{i\theta_2}}|_{\clm_2} U,
\]
if and only if either (i) $\clm_1$ and $\clm_2$ are Type I, or (ii) $\clm_1$ and $\clm_2$ are Type II and
\[
\theta_1=\theta_2,
\]
and
\[
\frac{\sigma_1}{\sigma_2} = \sqrt{\frac{1+\|g_1\|^2}{1+\|g_2\|^2}},
\]
where
\[
\clm_j = \mathbb{C}\begin{bmatrix} g_j\\ 1\end{bmatrix} \oplus \left(\varphi_j H^2\oplus\{0\}\right),
\]
is the canonical representation of $\clm_j$, $j=1,2$.

Given such a pair of invariant subspaces $\clm_1$ and $\clm_2$ as above, we say that $\clm_1$ and $\clm_2$ are \textit{unitarily equivalent} if there exists a unitary operator $U$ satisfying the above intertwining relation. Similarly, one can define unitarily equivalent invariant subspaces of the shift operator on the Hardy space, the Bergman space, the Dirichlet space, and many more. Invariant subspaces of the shift on $H^2$ are always unitarily equivalent, whereas they are never unitarily equivalent for the Bergman or Dirichlet spaces \cite{Richter}. From this perspective, the Brownian shift exhibits a mixture of invariant subspaces—a property that is highly distinctive compared to other classical operators (see the examples in Section \ref{sec: examp}). Additionally, we will soon see in \eqref{eqn: pert} that these operators are rank-one perturbations of an isometry. We will further comment on these observations at the end of this paper.

Now we turn to an asymptotic property of Brownian shifts. Let $T$ be a contraction on a Hilbert space $\clh$. We say that $T$ is pure, denoted by $T \in C_{\cdot 0}$, if
\[
SOT-\lim_{m \raro \infty} T^{*m}= 0.
\]
Furthermore, we say that $T$ satisfies the $C_{00}$-property, which we simply write as $T \in C_{00}$, if both $T$ and $T^*$ are pure.

Operators in class $C_{\cdot 0}$ or $C_{0 0}$ are of interest. The asymptotic property is often useful in representing these operators \cite{NaFo70}. Examples of operators in $C_{0 0}$ include strict contractions. Moreover, any operator can be scaled by a scalar so that the resulting operator belongs to $C_{0 0}$. However, scaling an operator is not always a desirable method for revealing its structure. In our present context, we first prove that Brownian shifts are not power-bounded and, hence, in particular, are not even similar to contractions. Next, we highlight a peculiar property of Brownian shifts: namely, we prove in Theorem \ref{Brown_C00} that for any covariance $\sigma > 0$ and angle $\theta \in [0, 2 \pi)$, the normalized operator
\[
\frac{1}{\|\br\|} \br \in C_{00}.
\]

There are many reasons to study Brownian shifts, as also pointed out by Agler and Stankus in their paper \cite{Agler-Stankus}. For instance, Brownian shifts play a crucial role in understanding the structure of $2$-isometries, a notion introduced by Agler decades ago (cf. \cite{Bill}). We refer the reader to \cite{Agler-Stankus} for representations of $2$-isometries and to \cite{Jan} for some recent developments. In addition, we highlight that a Brownian shift can be thought of as a rank-one perturbation of an isometry. Indeed:
\begin{equation}\label{eqn: pert}
\br = B_s + R,
\end{equation}
where
\[
B_s =\begin{bmatrix}S & 0\\0 & 1\end{bmatrix},
\]
is an isometry, and
\[
R = \begin{bmatrix}0 & \sigma (1\otimes 1)\\0 & e^{i\theta}-1\end{bmatrix},
\]
is a rank one operator on $H^2 \oplus \mathbb{C}$. The theory of perturbed operators and their invariant subspaces is certainly of interest. From this perspective, the result of Agler and Stankus on the invariant subspaces of Brownian shifts is particularly notable. Subsequently, the present work aims to shed new light on the general theory of operators and functions for Brownian shifts.

The paper is organized as follows. In Section \ref{sec: Unit equiv}, we present a complete description of unitarily equivalent invariant subspaces of Brownian shifts. In Section \ref{sec: power bdd}, we prove that normalized Brownian shifts are always in $C_{00}$. In the final section, Section \ref{sec: examp}, we illustrate our results with some concrete examples and point out that Brownian shifts are irreducible.

\section{Unitary equivalence}\label{sec: Unit equiv}

Recall that the nonzero invariant subspaces of $S$ in $H^2$ are given by $\vp H^2$, where $\vp$ runs over all inner functions from $H^\infty$ \cite{Beurling}. Given a pair of $S$-invariant subspaces $\vp_1 H^2$ and $\vp_2 H^2$ for some inner functions $\vp_1, \vp_2 \in H^\infty$, we consider the restrictions $S|_{\vp_1 H^2}$ and $S|_{\vp_2 H^2}$ of $S$ on $\vp_1 H^2$ and $\vp_2 H^2$, respectively. It is now easy to see that there exists a unitary operator $U: \vp_1 H^2 \raro \vp_2 H^2$ such that
\[
U S|_{\vp_1 H^2} = S|_{\vp_2 H^2} U.
\]
In short, we denote such a unitary equivalence property as
\[
S|_{\vp_1 H^2} \cong S|_{\vp_2 H^2}.
\]
Therefore, as far as operators are concerned, restrictions of $S$ on its invariant subspaces do not yield anything new. This prompts the question of distinguishing the restrictions of Brownian shifts on their invariant subspaces. We remind the reader that the invariant subspaces of Brownian shifts are also described by inner functions. For Type II invariant subspaces, these inner functions must additionally attain values at points on the circle corresponding to the given angles of the associated Brownian shifts. Also, recall that given a Type II invariant subspace $\clm$ of a Brownian shift, the canonical representation of $\clm$, as given in the introductory section (Section \ref{sec intro}), implies that the function $g$ is in $\clk_\vp$, where
\[
\clk_\vp = H^2 \ominus \vp H^2,
\]
is the \textit{model space} \cite{NaFo70}. This yields the useful relation
\[
g \perp \vp H^2.
\]
We now investigate the unitary equivalence of the invariant subspaces of Brownian shifts.

\begin{theorem}\label{thm: unit equiv}
Fix angles $\theta_1, \theta_2 \in [0, 2 \pi)$ and covariances $\sigma_1, \sigma_2 > 0$. Let $\clm_1$ and $\clm_2$ be nonzero closed invariant subspaces of the Brownian shifts $B_{\sigma_1, e^{i\theta_1}}$ and $B_{\sigma_2, e^{i\theta_2}}$, respectively. Then
\[
B_{\sigma_1, e^{i\theta_1}}{\big|_{\clm_1}} \cong B_{\sigma_2, e^{i\theta_2}}{\big|_{\clm_2}},
\]
if and only if any one of the following conditions is true:
\begin{itemize}
\item[(1)]
Both $\clm_1$ and $\clm_2$ are Type I.
\item[(2)]
Both $\clm_1$ and $\clm_2$ are Type II, along with the facts that
\[
\theta_1 = \theta_2,
\]
and
\[
\sigma_2^2(1+\|g_1\|^2)=\sigma_1^2(1+\|g_2\|^2),
\]
where $\clm_j = \mathbb{C}\begin{bmatrix}g_j\\ 1\end{bmatrix} \oplus \left(\varphi_j H^2\oplus\{0\}\right)$ is the canonical representation of $\clm_j$, and $g_j = \sigma_j\left(\frac{\overline{\vp_j(e^{i\theta_j})} \vp_j - 1}{z - e^{i\theta_j}}\right)$ for $j=1,2$.
\end{itemize}
\end{theorem}
\begin{proof}
Throughout the proof, we consider $\mu_{\vp_j} \in \mathbb{T}$ as
\[
\mu_{\vp_j} = \overline{\vp_j(e^{i\theta_j})},
\]
for $j=1, 2$. Let us start with the proof of the ``if'' part. Provided that condition (1) holds, we assume $\clm_1=\varphi H^2\oplus\{0\}$ and $\clm_2=\psi H^2\oplus\{0\}$ for some inner functions $\varphi$ and $\psi$ from $H^\infty$, and consider the unitary operator $U:\clm_1\to \clm_2$, defined by
\[
U\begin{bmatrix} \varphi h\\ 0\end{bmatrix}=\begin{bmatrix} \psi h\\ 0\end{bmatrix}
\]
for all $h\in H^2$. Then
\begin{align*}
B_{\sigma_2, e^{i\theta_2}}U\begin{bmatrix} \varphi h\\ 0\end{bmatrix}&=\begin{bmatrix} S & \sigma_2(1\otimes 1)\\ 0 & e^{i\theta_2}\end{bmatrix}\begin{bmatrix} \psi h\\ 0\end{bmatrix}
= \begin{bmatrix} z\psi h\\ 0\end{bmatrix}
= U\begin{bmatrix} z\varphi h\\ 0\end{bmatrix}
= UB_{\sigma_1, e^{i\theta_1}}\begin{bmatrix} \varphi h\\ 0\end{bmatrix},
\end{align*}
which implies $B_{\sigma_1, e^{i\theta_1}}{\big|_{\clm_1}} \cong B_{\sigma_2, e^{i\theta_2}}{\big|_{\clm_2}}$. This part is the same as the equivalence of invariant subspaces of $S$ on $H^2$. Next, we assume that (2) is true. Set
\[
\theta=\theta_1=\theta_2.
\]
Define a linear operator $U:\clm_1\to \clm_2$ by
\[
U\left(c_1\mu_{\varphi_1}\begin{bmatrix}\varphi_1 h \\ 0\end{bmatrix} + c_2\sigma_2 \begin{bmatrix} g_1\\ 1 \end{bmatrix}\right)
= c_1\mu_{\varphi_2}\begin{bmatrix}\varphi_2 h \\ 0\end{bmatrix} + c_2 \sigma_1 \begin{bmatrix} g_2\\ 1\end{bmatrix},
\]
for all $h\in H^2$ and scalars $c_1, c_2\in\mathbb{C}$. It is evident from the construction that $U$ is surjective. Moreover, as $\sigma_2^2(1+\|g_1\|^2)=\sigma_1^2(1+\|g_2\|^2)$, we have
\begin{align*}
\left\|c_1\mu_{\varphi_2}\begin{bmatrix}\varphi_2 h \\ 0\end{bmatrix} + c_2 \sigma_1\begin{bmatrix} g_2\\ 1\end{bmatrix}\right\|^2
&=|c_1|^2\left\|\begin{bmatrix}\varphi_2 h \\ 0\end{bmatrix}\right\|^2+\sigma_1^2|c_2|^2\left\|\begin{bmatrix}g_2 \\ 1\end{bmatrix}\right\|^2
\\
&=|c_1|^2\|h\|^2+\sigma_1^2|c_2|^2(\|g_2\|^2+1)
\\
&=|c_1|^2\|h\|^2+\sigma_2^2|c_2|^2(\|g_1\|^2+1)
\\
&=\left\|c_1\mu_{\varphi_1}\begin{bmatrix}\varphi_1 h \\ 0\end{bmatrix}+c_2\sigma_2\begin{bmatrix} g_1\\ 1\end{bmatrix}\right\|^2.
\end{align*}
This implies that $U$ is an isometry, and therefore, is a unitary operator as well. For notational simplicity, set
\[
F = c_1\mu_{\varphi_1}\begin{bmatrix}\varphi_1 h \\ 0\end{bmatrix} + c_2 \sigma_2 \begin{bmatrix} g_1\\ 1\end{bmatrix}.
\]
We observe that
\begin{align*}
B_{\sigma_2, e^{i\theta}}U F
&=\begin{bmatrix} S & \sigma_2(1\otimes 1)\\ 0 & e^{i\theta}\end{bmatrix}\left(c_1\mu_{\varphi_2}\begin{bmatrix}\varphi_2 h \\ 0\end{bmatrix}+c_2\sigma_1\begin{bmatrix} g_2\\ 1\end{bmatrix}\right)
\\
&=c_1\mu_{\varphi_2}\begin{bmatrix}z\varphi_2 h \\ 0\end{bmatrix}+c_2\sigma_1\begin{bmatrix} zg_2+\sigma_2\\ e^{i\theta}\end{bmatrix}.
\end{align*}
At this point, we recall the representation of $g_j$:
\[
g_j(z) = \sigma_j \left(\frac{\mu_{\vp_j} \varphi_j - 1}{z-e^{i\theta}}\right),
\]
where $\mu_{\vp_j} = \overline{\vp_j(e^{i\theta})}$ for $j =1,2$. This yields
\[
z g_j + \sigma_j = e^{i\theta} g_j + \sigma_j \mu_{\varphi_j}\varphi_j,
\]
for $j=1,2$. Therefore, we have
\begin{align*}
B_{\sigma_2, e^{i\theta}}U F & = c_1\mu_{\varphi_2}\begin{bmatrix}z\varphi_2 h \\ 0\end{bmatrix}+c_2\sigma_1\left(e^{i\theta}\begin{bmatrix} g_2\\ 1\end{bmatrix}+\begin{bmatrix}\sigma_2 \mu_{\varphi_2}\varphi_2\\ 0\end{bmatrix}\right)
\\
&=U\left(c_1\mu_{\varphi_1}\begin{bmatrix}z\varphi_1 h \\ 0\end{bmatrix}+c_2e^{i\theta}\sigma_2\begin{bmatrix} g_1\\ 1\end{bmatrix}+c_2\sigma_1\begin{bmatrix}\sigma_2 \mu_{\varphi_1}\varphi_1\\ 0\end{bmatrix}\right)
\\
&=U\left(c_1\mu_{\varphi_1}\begin{bmatrix}z\varphi_1 h \\ 0\end{bmatrix}+c_2\sigma_2\begin{bmatrix} zg_1+\sigma_1\\ e^{i\theta}\end{bmatrix}\right)\\
&=UB_{\sigma_1, e^{i\theta}}\left(c_1\mu_{\varphi_1}\begin{bmatrix}\varphi_1 h \\ 0\end{bmatrix}+c_2\sigma_2\begin{bmatrix} g_1\\ 1\end{bmatrix}\right),
\end{align*}
that is, $B_{\sigma_2, e^{i\theta}}U F = UB_{\sigma_1, e^{i\theta}} F$, which again ensures that $B_{\sigma_1, e^{i\theta_1}}{\big|_{\clm_1}} \cong B_{\sigma_2, e^{i\theta_2}}{\big|_{\clm_2}}$. Conversely, let us start by showing that if we assume, without loss of generality, that $\clm_1=\varphi H^2\oplus\{0\}$ is Type I and $\clm_2= \mathbb{C}\begin{bmatrix} g \\ 1\end{bmatrix} \oplus (\psi H^2\oplus\{0\})$ is Type II, for certain inner functions $\varphi$ and $\psi$, then $B_{\sigma_1, e^{i\theta_1}}{\big|_{\clm_1}}$ and $B_{\sigma_2, e^{i\theta_2}}{\big|_{\clm_2}}$ are not unitarily equivalent. Indeed, if it is not true, then, in particular, we will have the norm identity
\[
\left\|B_{\sigma_1, e^{i\theta_1}}{\big|_{\clm_1}}\right\|= \left\|B_{\sigma_2, e^{i\theta_2}}{\big|_{\clm_2}}\right\|.
\]
However, for any $h\in H^2$, we have
\begin{align*}
\left\|B_{\sigma_1, e^{i\theta_1}}\begin{bmatrix} \varphi h\\ 0\end{bmatrix}\right\| =\left\|\begin{bmatrix} z\varphi h\\ 0\end{bmatrix}\right\|
=\left\|\begin{bmatrix} \varphi h\\ 0\end{bmatrix}\right\|,
\end{align*}
that is, $\left\|B_{\sigma_1, e^{i\theta_1}}{\big|_{\clm_1}}\right\|=1$. On the other hand, we have
\begin{align*}
\left\|B_{\sigma_2, e^{i\theta_2}}\begin{bmatrix} g\\ 1\end{bmatrix}\right\|^2 =\left\|\begin{bmatrix} zg+\sigma_2\\ e^{i\theta_2}\end{bmatrix}\right\|^2
= 1+\|g\|^2+\sigma_2^2
= \left\|\begin{bmatrix} g\\
1\end{bmatrix}\right\|^2+\sigma_2^2
& >\left\|\begin{bmatrix} g\\
1\end{bmatrix}\right\|^2,
\end{align*}
and hence
\[
\left\|B_{\sigma_2, e^{i\theta_2}}{\big|_{\clm_2}}\right\|>1,
\]
which leads to a contradiction. Therefore, $\clm_1$ and $\clm_2$ must be of same type. It is therefore enough to assume that both $\clm_1$ and $\clm_2$ are of Type II. Consider the canonical representations of $\clm_j$ as
\[
\clm_j = \mathbb{C}\begin{bmatrix}g_j\\ 1\end{bmatrix} \oplus \left(\varphi_j H^2\oplus\{0\}\right),
\]
where $\vp_j \in H^\infty$ is an inner function, and $g_j = \sigma_j\left(\frac{\overline{\vp_j(e^{i\theta_j})} \vp_j - 1}{z - e^{i\theta_j}}\right)$ for $j=1,2$. Also suppose $U: \clm_1\to \clm_2$ is the unitary operator satisfying the intertwining relation
\[
U B_{\sigma_1, e^{i\theta_1}}{\big|_{\clm_1}} = B_{\sigma_2, e^{i\theta_2}}{\big|_{\clm_2}} U.
\]
Let
\[
U \begin{bmatrix} g_1 \\ 1\end{bmatrix} = \begin{bmatrix} \varphi_2 h_2 \\ 0\end{bmatrix}+\alpha\begin{bmatrix} g_2 \\ 1\end{bmatrix},
\]
for some $h_2\in H^2$ and $\alpha\in\mathbb{C}$. As $\left\|U\begin{bmatrix} g_1 \\ 1\end{bmatrix}\right\|^2 = \|g_1\|^2+1$ and $\|\varphi_2 h_2\| = \|h_2\|$, it follows that
\[
\|g_1\|^2+1 = \left\|\begin{bmatrix} \varphi_2 h_2 \\ 0\end{bmatrix}\right\|^2+|\alpha|^2\left\|\begin{bmatrix} g_2 \\ 1\end{bmatrix}\right\|^2 =\|h_2\|^2+|\alpha|^2\left(1 + \|g_2\|^2\right),
\]
that is,
\begin{equation}\label{Brown_5}
(1 + \|g_1\|^2) - |\alpha|^2 (1 + \|g_2\|^2) = \|h_2\|^2.
\end{equation}
Also, since
\[
B_{\sigma_1, e^{i\theta_1}} \begin{bmatrix} g_1 \\ 1\end{bmatrix} = \begin{bmatrix} zg_1+\sigma_1\\ e^{i\theta_1}\end{bmatrix},
\]
and $\|zg_1+\sigma_1\|^2 = \|g_1\|^2 + \sigma_1^2$ (as $\sigma_1 \in \mathbb{C}$, $zg_1 \in z H^2$, and $\mathbb{C} \perp z H^2$), we have
\[
\left\|UB_{\sigma_1, e^{i\theta_1}}\begin{bmatrix}g_1 \\ 1\end{bmatrix}\right\|^2 = \left\|\begin{bmatrix} zg_1+\sigma_1\\ e^{i\theta_1}\end{bmatrix}\right\|^2 = \|g_1\|^2 + \sigma_1^2 + 1.
\]
Then, $U B_{\sigma_1, e^{i\theta_1}}{\big|_{\clm_1}} = B_{\sigma_2, e^{i\theta_2}}{\big|_{\clm_2}} U$ implies
\[
\begin{aligned}
\|g_1\|^2+\sigma_1^2+1 & = \left\|B_{\sigma_2, e^{i\theta_2}}U \begin{bmatrix}g_1 \\ 1\end{bmatrix}\right\|^2
\\
& = \left\|B_{\sigma_2, e^{i\theta_2}} \left(\begin{bmatrix} \varphi_2 h_2 \\ 0\end{bmatrix}+\alpha\begin{bmatrix} g_2 \\ 1\end{bmatrix} \right) \right\|^2
\\
&=\left\|\begin{bmatrix} z\varphi_2 h_2 \\ 0\end{bmatrix}\right\|^2+|\alpha|^2\left\|\begin{bmatrix} zg_2+\sigma_2 \\ e^{i\theta_2}\end{bmatrix}\right\|^2
\\
&=\|h_2\|^2+|\alpha|^2\left(\|g_2\|^2+\sigma_2^2+1\right).
\end{aligned}
\]
As we already know from \eqref{Brown_5} that $1 + \|g_1\|^2 = \|h_2\|^2+|\alpha|^2\left(1 + \|g_2\|^2 \right)$, we conclude
\begin{equation}\label{Brown_7}
\sigma_1^2=|\alpha|^2\sigma_2^2.
\end{equation}
As $\begin{bmatrix} \varphi_2 h_2\\ 0\end{bmatrix} \in \clm_2$, there exist $h_1\in H^2$ and scalar $\beta$ such that
\[
U^*\begin{bmatrix} \varphi_2 h_2\\ 0\end{bmatrix} = \begin{bmatrix} \varphi_1 h_1 \\ 0\end{bmatrix}+\beta\begin{bmatrix} g_1\\ 1\end{bmatrix}.
\]
Since
\[
\|h_2\|^2 = \| \vp_2 h_2\|^2 = \left\|U^*\begin{bmatrix} \varphi_2 h_2 \\ 0\end{bmatrix}\right\|^2,
\]
we conclude that
\[
\|h_2\|^2 = \left\|\begin{bmatrix} \varphi_1 h_1 \\ 0\end{bmatrix}\right\|^2+|\beta|^2\left\|\begin{bmatrix} g_1 \\ 1\end{bmatrix}\right\|^2
=\|h_1\|^2+|\beta|^2\left(1 + \|g_1\|^2\right).
\]
On the other hand,
\[
\|h_2\|^2 = \|z\varphi_2 h_2\|^2 = \left\|\begin{bmatrix} z\varphi_2 h_2\\ 0\end{bmatrix}\right\|^2 = \left\|B_{\sigma_2, e^{i\theta_2}} \begin{bmatrix}\varphi_2 h_2 \\ 0\end{bmatrix}\right\|^2 = \left\|B_{\sigma_2, e^{i\theta_2}}UU^*\begin{bmatrix}\varphi_2 h_2 \\ 0\end{bmatrix}\right\|^2.
\]
Then by the intertwining relation $U B_{\sigma_1, e^{i\theta_1}}{\big|_{\clm_1}} = B_{\sigma_2, e^{i\theta_2}}{\big|_{\clm_2}} U$ and the representation of $U^*\begin{bmatrix} \varphi_2 h_2\\ 0\end{bmatrix}$ above, we have
\[
\begin{aligned}
\|h_2\|^2 & = \left\|U B_{\sigma_1, e^{i\theta_1}}\left(\begin{bmatrix} \varphi_1 h_1 \\ 0\end{bmatrix}+\beta\begin{bmatrix} g_1\\ 1\end{bmatrix}\right)\right\|^2
\\
& = \left\|B_{\sigma_1, e^{i\theta_1}} \begin{bmatrix} \varphi_1 h_1 \\ 0\end{bmatrix} + B_{\sigma_1, e^{i\theta_1}} \beta\begin{bmatrix} g_1\\ 1\end{bmatrix}\right\|^2
\\
& = \left\|\begin{bmatrix} z\varphi_1 h_1 \\ 0\end{bmatrix}\right\|^2+|\beta|^2\left\|\begin{bmatrix} zg_1+\sigma_1 \\ e^{i\theta_1}\end{bmatrix}\right\|^2
\\
& = \|h_1\|^2+|\beta|^2\left(\|g_1\|^2+\sigma_1^2+1\right).
\end{aligned}
\]
Combining this with $\|h_2\|^2 = \|h_1\|^2+|\beta|^2\left(1 + \|g_1\|^2\right)$, we obtain
\[
\beta=0,
\]
and hence, the action of $U^*$ on the vector $\begin{bmatrix} \varphi_2 h_2\\ 0\end{bmatrix}$ reduces to $U^*\begin{bmatrix} \varphi_2 h_2\\ 0\end{bmatrix} = \begin{bmatrix} \varphi_1 h_1 \\ 0\end{bmatrix}$, and consequently
\[
U \begin{bmatrix} \varphi_1 h_1 \\ 0\end{bmatrix} = \begin{bmatrix} \varphi_2 h_2\\ 0\end{bmatrix}.
\]
As a result, the identity $U\begin{bmatrix} g_1 \\ 1\end{bmatrix} = \begin{bmatrix} \varphi_2 h_2 \\ 0\end{bmatrix} + \alpha\begin{bmatrix} g_2 \\ 1\end{bmatrix}$ implies
\[
U\begin{bmatrix} g_1 \\ 1\end{bmatrix}=U\begin{bmatrix} \varphi_1 h_1 \\ 0\end{bmatrix}+\alpha\begin{bmatrix} g_2 \\ 1\end{bmatrix},
\]
that is,
\[
U\begin{bmatrix} g_1 - \varphi_1 h_1 \\ 1\end{bmatrix} = \alpha\begin{bmatrix} g_2 \\ 1\end{bmatrix}.
\]
Recall that $g_1 \in \clk_{\vp_1}$ (the model space corresponding to the inner function $\vp_1$ in $H^\infty$). Then $\vp_1 h_1 \in \clk_{\vp_1}^\perp = \vp_1 H^2$ yields
\[
\|g_1 - \vp_1 h_1\|^2 = \|g_1\|^2 + \|\vp_1 h_1\|^2 = \|g_1\|^2 + \|h_1\|^2,
\]
and so
\[
1+\|g_1\|^2+\|h_1\|^2 =\left\|U\begin{bmatrix} g_1-\varphi_1h_1\\ 1\end{bmatrix}\right\|^2 = |\alpha|^2\left\|\begin{bmatrix} g_2 \\ 1\end{bmatrix}\right\|^2,
\]
which implies
\[
1+\|g_1\|^2+\|h_1\|^2 = |\alpha|^2(1+\|g_2\|^2).
\]
But, by \eqref{Brown_5}, we know that $|\alpha|^2\left(\|g_2\|^2+1\right) = \|g_1\|^2-\|h_2\|^2 + 1$. Therefore, $\|h_1\|^2 + \|h_2\|^2=0$, yielding that
\[
h_1 = h_2 = 0.
\]
The identity \eqref{Brown_5} then changes to $(1 + \|g_1\|^2) = |\alpha|^2 (1 + \|g_2\|^2)$, and then, \eqref{Brown_7} yields the desired identity
\[
\sigma_2^2(1+\|g_1\|^2)=\sigma_1^2(1+\|g_2\|^2).
\]
It remains to show that $\theta_1 = \theta_2$. First, we use the intertwining property $U B_{\sigma_1, e^{i\theta_1}}{\big|_{\clm_1}} = B_{\sigma_2, e^{i\theta_2}}{\big|_{\clm_2}} U$ to observe that
\[
\left\langle B_{\sigma_1, e^{i\theta_1}}\begin{bmatrix} g_1 \\ 1\end{bmatrix}, U^*\begin{bmatrix} g_2 \\ 1\end{bmatrix}\right\rangle= \left\langle B_{\sigma_2, e^{i\theta_2}}U\begin{bmatrix} g_1 \\ 1\end{bmatrix}, \begin{bmatrix} g_2 \\ 1\end{bmatrix}\right\rangle.
\]
As $U\begin{bmatrix} g_1 \\ 1\end{bmatrix}=\alpha\begin{bmatrix} g_2 \\ 1\end{bmatrix}$, we have $U^* \begin{bmatrix} g_2 \\ 1\end{bmatrix} = \frac{1}{\alpha} \begin{bmatrix} g_1 \\ 1\end{bmatrix}$ (we already know that $\alpha \neq 0$), and consequently, the above identity implies
\[
\frac{1}{\overline{\alpha}}\left\langle \begin{bmatrix} zg_1+\sigma_1\\ e^{i\theta_1}\end{bmatrix}, \begin{bmatrix} g_1 \\ 1\end{bmatrix}\right\rangle=\alpha\left\langle \begin{bmatrix} zg_2+\sigma_2\\ e^{i\theta_2}\end{bmatrix}, \begin{bmatrix} g_2 \\ 1\end{bmatrix}\right\rangle.
\]
Again, recall that
\[
z g_j + \sigma_j = e^{i\theta_j} g_j + \sigma_j \mu_{\varphi_j}\varphi_j,
\]
for $j=1,2$. Then
\[
\left\langle e^{i\theta_1}\begin{bmatrix} g_1\\ 1\end{bmatrix}+\begin{bmatrix}\sigma_1\mu_{\varphi_1}\varphi_1 \\ 0\end{bmatrix}, \begin{bmatrix} g_1\\ 1\end{bmatrix}\right\rangle = |\alpha|^2 \left\langle  e^{i\theta_2}\begin{bmatrix} g_2\\ 1\end{bmatrix}+\begin{bmatrix}\sigma_2\mu_{\varphi_2}\varphi_2 \\ 0\end{bmatrix}, \begin{bmatrix} g_2\\ 1\end{bmatrix} \right\rangle,
\]
and so
\[
e^{i\theta_1} \left\|\begin{bmatrix} g_1\\ 1\end{bmatrix}\right\|^2 = e^{i\theta_2} |\alpha|^2 \left\|\begin{bmatrix} g_2\\ 1\end{bmatrix}\right\|^2,
\]
that is,
\[
e^{i\theta_1}\left(1+\|g_1\|^2\right) = e^{i\theta_2}|\alpha|^2(1+\|g_2\|^2).
\]
As we already know that $\left(1+\|g_1\|^2\right) = |\alpha|^2(1+\|g_2\|^2)$, it follows that $e^{i\theta_1}=e^{i\theta_2}$. Since $\theta_1, \theta_2\in[0, 2\pi)$, we finally conclude that $\theta_1=\theta_2$.
\end{proof}

Therefore, in contrast to the shift operators on the Hardy space, the invariant subspaces of the Brownian shifts have the potential to lead to different operators. In the final section of this paper, we illustrate these results with concrete examples.

Finally, note that by setting $\vp_2=1$, $g_2=0$ in the above theorem, we get that for any Type II invariant subspace $\clm_1$ of $B_{\sigma_1, e^{i\theta_1}}$,
$$
B_{\sigma_1, e^{i\theta_1}}{\big|_{\clm_1}} \cong B_{\frac{\sigma_1}{\sqrt{1+\|g_1\|^2}}, e^{i\theta_1}}.
$$
This was previously observed by Agler and Stankus in \cite[Proposition 5.75]{Agler-Stankus}.

\section{$\frac{1}{\sqrt{1+\sigma^2}}\br \in C_{00}$}\label{sec: power bdd}

The aim of this section is to prove that Brownian shifts, when scaled by their reciprocal norms, belong to $C_{00}$. Scaling an operator by its reciprocal norm does turn it into a contraction, but does not necessarily place it in the class $C_{00}$ (nor even $C_{\cdot 0}$). Simply consider a unitary operator or the shift. This is where Brownian shifts exhibit different behavior.

We first prove that Brownian shifts are not even similar to contractions, and we establish this by showing that they are not power-bounded. Recall that a bounded linear operator $A$ acting on a Hilbert space $\clh$ is said to be \textit{power bounded} if the sequence of real numbers
\[
\{\|A^n\|\}_{n=1}^\infty,
\]
is bounded. It is evident that any bounded linear operator similar to a contraction is power-bounded; however, the converse is far from being true. Let us fix a Brownian shift $\br$. Observe that
\[
\br\begin{bmatrix}
    0 \\ 1
\end{bmatrix} =\begin{bmatrix}
    S & \sigma(1\otimes 1)\\
    0 & e^{i\theta}
\end{bmatrix}\begin{bmatrix}
    0 \\ 1
\end{bmatrix}
=\begin{bmatrix}
\sigma \\ e^{i\theta}
\end{bmatrix}.
\]
In general, by the principle of mathematical induction, we conclude that
\begin{equation}\label{Similar_eq1}
B_{\sigma, e^{i\theta}}^m\begin{bmatrix}
0 \\ 1
\end{bmatrix}
=\begin{bmatrix}
\sigma\sum_{k=0}^{m-1}e^{ik\theta}z^{m-k-1} \\ e^{im\theta}
\end{bmatrix},
\end{equation}
for all $m \geq 1$. Using the norm of functions in $H^2$, we conclude that
\[
\left\|\begin{bmatrix}
\sigma\sum_{k=0}^{m-1}e^{ik\theta}z^{m-k-1} \\ e^{im\theta}
\end{bmatrix}\right\|^2 = 1+m\sigma^2.
\]
As $\left\|\begin{bmatrix} 0 \\ 1\end{bmatrix} \right\| = 1$, it follows that
\[
\left\|B_{\sigma, e^{i\theta}}^m\right\|^2\geq\left\|B_{\sigma, e^{i\theta}}^m\begin{bmatrix}
    0 \\ 1
\end{bmatrix}\right\|^2=1+m\sigma^2 \raro \infty,
\]
as $m\to\infty$. This proves the following:

\begin{proposition}\label{Brown_similar}
$\br$ on $H^2\oplus\mathbb{C}$ is not power bounded.
\end{proposition}

This in particular shows that $\br$ is not similar to contractions. However, the following is true:

\begin{theorem}\label{Brown_C00}
$\frac{1}{\|\br\|}\br \in C_{00}$ for all covariance $\sigma > 0$ and angle $\theta \in [0, 2\pi)$.
\end{theorem}
\begin{proof}
Let us begin by computing the norm of $\br$. For any $\begin{bmatrix}f \\ \alpha\end{bmatrix}\in H^2\oplus\mathbb{C}$ with unit norm $\|f\|^2+|\alpha|^2=1$, we have, in particular, that $|\alpha|\leq 1$. Moreover,
\[
\left\|\br\begin{bmatrix} f \\ \alpha
\end{bmatrix}\right\|^2=\left\|\begin{bmatrix} zf+\sigma\alpha \\ e^{i\theta}\alpha
\end{bmatrix}\right\|^2=\|zf+\sigma\alpha\|^2+|\alpha|^2=1+\sigma^2|\alpha|^2\leq 1+\sigma^2,
\]
and equality occurs for $f=0, \alpha=1$. Therefore, we have the norm of $\br$ as
\[
\|\br\|=\sqrt{1+\sigma^2}.
\]
Consequently, the operator
\[
\tilde{B}:=\frac{1}{\sqrt{1+\sigma^2}}\br,
\]
becomes a contraction on $H^2 \oplus \mathbb{C}$. Pick $u \in H^2 \oplus \mathbb{C}$ and write
\[
u=c_0\begin{bmatrix}
    0 \\ 1
\end{bmatrix}+\sum_{k=0}^\infty c_{k+1}\begin{bmatrix}
    z^k \\ 0
\end{bmatrix}\in H^2\oplus\mathbb{C},
\]
for some $c_k \in \mathbb{C}$ for $k \geq 0$. Making use of (\ref{Similar_eq1}), a little computation reveals, for each $n \geq 1$, that
\begin{align*}
\left\|\tilde{B}^n u\right\|^2&=\frac{1}{(1+\sigma^2)^n}\left\|c_0 B_{\sigma, e^{i\theta}}^n\begin{bmatrix}
    0 \\ 1
\end{bmatrix}+\sum_{k=0}^\infty c_{k+1} B_{\sigma, e^{i\theta}}^n\begin{bmatrix}
    z^k \\ 0
\end{bmatrix}\right\|^2\\
&=\frac{1}{(1+\sigma^2)^n}\left\|c_0\begin{bmatrix}
    \sigma\sum_{k=0}^{n-1}e^{ik\theta}z^{n-k-1} \\ e^{in\theta}
\end{bmatrix}+\sum_{k=0}^\infty c_{k+1}\begin{bmatrix}
    z^{n+k} \\ 0
\end{bmatrix}\right\|^2\\
&= \frac{1}{(1+\sigma^2)^n}\left(|c_0|^2\left\|\begin{bmatrix}
    \sigma\sum_{k=0}^{n-1}e^{ik\theta}z^{n-k-1} \\ e^{in\theta}
\end{bmatrix}\right\|^2+\sum_{k=0}^\infty|c_{k+1}|^2\left\|\begin{bmatrix}
    z^{n+k} \\ 0
\end{bmatrix}\right\|^2\right)\\
&=\frac{1}{(1+\sigma^2)^n}\left(|c_0|^2(1+n\sigma^2)+\sum_{k=0}^\infty|c_{k+1}|^2\right),
\end{align*}
and hence
\[
\left\|\tilde{B}^n u\right\|^2 \leq\frac{\|u\|^2+n|c_0|^2\sigma^2}{(1+\sigma^2)^n} \leq \frac{2}{\sigma^4}\left(\frac{\|u\|^2}{n(n-1)}+\frac{\sigma^2|c_0|^2}{n-1}\right) \longrightarrow 0,
\]
as $n\to\infty$. This implies that $\tilde{B}^* \in C_{.0}$. On the other hand, we have
\[
{\tilde{B}}^{*n}\begin{bmatrix}
    0 \\ 1
\end{bmatrix}
=\frac{1}{(1+\sigma^2)^{\frac{n}{2}}}B_{\sigma, e^{i\theta}}^{*^n}\begin{bmatrix}
    0 \\ 1
\end{bmatrix}=\frac{e^{-in\theta}}{(1+\sigma^2)^{\frac{n}{2}}}\begin{bmatrix}
    0 \\ 1
\end{bmatrix},
\]
for all $n \geq 1$. Moreover, if $0\leq k< n$, then
\[
{\tilde{B}}^{*n} \begin{bmatrix}
    z^k \\ 0
\end{bmatrix} =\frac{1}{(1+\sigma^2)^{\frac{n}{2}}}B_{\sigma, e^{i\theta}}^{*^n}\begin{bmatrix}
    z^k \\ 0
\end{bmatrix}\\ =\frac{1}{(1+\sigma^2)^{\frac{n}{2}}}B_{\sigma, e^{i\theta}}^{*^{n-k}}\begin{bmatrix}
    1 \\ 0
\end{bmatrix}\\ =\frac{\sigma}{(1+\sigma^2)^{\frac{n}{2}}}B_{\sigma, e^{i\theta}}^{*^{n-k-1}}\begin{bmatrix}
    0 \\ 1
\end{bmatrix},
\]
that is,
\[
{\tilde{B}}^{*n}\begin{bmatrix}
    z^k \\ 0
\end{bmatrix} =\frac{\sigma e^{-i(n-k-1)\theta}}{(1+\sigma^2)^{\frac{n}{2}}}\begin{bmatrix}
    0 \\ 1
\end{bmatrix}.
\]
Finally, for $k\geq n$, we have
$$
\tilde{B}^{*n}\begin{bmatrix}
    z^k \\ 0
\end{bmatrix}
=\frac{1}{(1+\sigma^2)^{\frac{n}{2}}}B_{\sigma, e^{i\theta}}^{*^n}\begin{bmatrix}
    z^k \\ 0
\end{bmatrix}
=\frac{1}{(1+\sigma^2)^{\frac{n}{2}}}\begin{bmatrix}
    z^{k-n} \\ 0
\end{bmatrix}.
$$
Therefore, we have
\begin{align*}
\left\|\tilde{B}^{*n} u\right\|^2&=\left\|c_0 \tilde{B}^{*n} \begin{bmatrix}
    0 \\ 1
\end{bmatrix} +\sum_{k=0}^{n-1} c_{k+1} \tilde{B}^{*n}\begin{bmatrix}
    z^k \\ 0
\end{bmatrix}+\sum_{k=n}^\infty c_{k+1} \tilde{B}^{*n}\begin{bmatrix}
    z^k \\ 0
\end{bmatrix}\right\|^2\\
& = \frac{1}{(1+\sigma^2)^n}\left\|\left(c_0 e^{-in\theta}+\sigma \sum_{k=0}^{n-1} c_{k+1}e^{-i(n-k-1)\theta}\right)\begin{bmatrix} 0 \\ 1
\end{bmatrix}+\sum_{k=n}^\infty c_{k+1} \begin{bmatrix}
    z^{k-n} \\ 0
\end{bmatrix}\right\|^2\\
&=\frac{1}{(1+\sigma^2)^n}\left(\left|c_0 e^{-in\theta}+\sigma \sum_{k=0}^{n-1} c_{k+1}e^{-i(n-k-1)\theta}\right|^2+\sum_{k=n}^\infty |c_{k+1}|^2\right)\\
&\leq \frac{1}{(1+\sigma^2)^n}\left((1+n\sigma^2) \left(\sum_{k=0}^n |c_{k}|^2\right) + \sum_{k=n}^\infty |c_{k+1}|^2\right)\\
&\leq \|u\|^2\frac{2+n\sigma^2}{(1+\sigma^2)^n}\\
&\leq\frac{2\|u\|^2}{\sigma^4}\left(\frac{2}{n(n-1)}+\frac{\sigma^2}{n-1}\right).
\end{align*}
But
\[
\frac{2}{n(n-1)}+\frac{\sigma^2}{n-1} \longrightarrow 0,
\]
as $n\to\infty$. As a result, $\tilde{B} \in C_{.0}$, which completes the proof of the theorem.
\end{proof}

As we have proved in the theorem above that $\|\br\|=\sqrt{1+\sigma^2}$, it follows that $\frac{1}{\sqrt{1+\sigma^2}} \br\in C_{00}$. In particular, if $\clm$ is an invariant subspace of $\br$, then the restriction operator
\[
\frac{1}{\sqrt{1+\sigma^2}} P_{\clm^\perp} \br|_{\clm^\perp} \in C_{\cdot 0}.
\]

We remark that similar convergence results have been studied in a broader framework, namely, $B$-operators, in \cite{Chavan et al}. Theorem 2.3 of \cite{Chavan et al}, when specialized to our setting, in particular implies that adjoints of non-normalized Brownian shifts are never in the $C_{\cdot 0}$-class.

\section{Examples}\label{sec: examp}

The purpose of this section is to illustrate the classification result, Theorem \ref{thm: unit equiv}, using concrete examples. We aim to specifically show that Theorem \ref{thm: unit equiv} indeed provides examples of unitarily as well as non-unitarily equivalent invariant subspaces of Brownian shifts. We also prove that Brownian shifts do not have nontrivial reducing subspaces.

We will present two examples, and Blaschke factors will play a role in both of them. For each $\alpha \in \D$, the \textit{Blaschke factor} $b_\alpha$ corresponding to $\alpha$ is defined by
\[
b_\alpha (z) = \frac{z-\alpha}{1-\bar\alpha z},
\]
for all $z \in \D$. Blaschke factors are the simplest examples of inner functions.

\begin{example}
For $\{\alpha_1, \alpha_2\} \subseteq (0,1)$, consider inner functions $\vp_j = b_{\alpha_j}$, $j=1,2$. Also, for each $\theta_1, \theta_2\in [0, 2\pi)$ and $\sigma_1, \sigma_2>0$, define $g_j \in H^2$ by
\[
g_j(z)=\sigma_j\frac{\varphi_j(z)\overline{\varphi_j(e^{i\theta_j})}-1}{z-e^{i\theta_j}},
\]
for $j=1,2$. It is easy to see that
\[
\|g_j\|^2=\sigma_j^2\left\|\frac{1-\alpha_j^2}{(1-\alpha_j e^{i\theta_j})(1-\alpha_j z)}\right\|^2=\sigma_j^2\frac{1-\alpha_j^2}{1+\alpha_j^2-2\alpha_j \cos \theta_j},
\]
for $j=1,2$. In particular, for the choice $\theta_1=\theta_2=0$, we have
\[
\|g_j\|^2=\sigma_j^2\frac{1+\alpha_j}{1-\alpha_j},
\]
and therefore, a little computation reveals that $\sigma_2^2(1+\|g_1\|^2)=\sigma_1^2(1+\|g_2\|^2)$ is satisfied, provided we have
\[
\frac{1}{\sigma_1^2}-\frac{1}{\sigma_2^2}=\frac{2(\alpha_2-\alpha_1)}{(1-\alpha_1)(1-\alpha_2)}.
\]
In particular, in view of Theorem \ref{thm: unit equiv}, if the pairs $\{\alpha_1, \alpha_2\}$ and $\{\sigma_1, \sigma_2\}$ fail to satisfy the above identity, then $B_{\sigma_1, 1}{\big|_{\clm_1}}$ and $B_{\sigma_2, 1}{\big|_{\clm_2}}$ are not unitarily equivalent, where
\[
\clm_j = \mathbb{C} \begin{bmatrix}g_j \\ 1\end{bmatrix} \oplus (\vp_j H^2 \oplus \{0\}),
\]
for $j=1,2$.
\end{example}

In the following example, we bring a singular inner function with a single atom.

\begin{example}
For $\alpha\in (0,1)$, consider the inner function $\vp_1 = b_\alpha$, and the other inner function as
\[
\varphi_2(z)=\exp\left(\frac{z+1}{z-1}\right),
\]
for all $z \in \D$. As usual, for each $\theta_1, \theta_2\in [0, 2\pi)$ and $\sigma_1, \sigma_2>0$, define $g_j: \D \raro \mathbb{C}$ by
\[
g_j(z)=\sigma_j\frac{\varphi_j(z)\overline{\varphi_j(e^{i\theta_j})}-1}{z-e^{i\theta_j}} \qquad (z \in \D),
\]
for $j=1,2$. Clearly, $g_1 \in H^2$. Let us first assume $\theta_1=\theta_2=\pi$. From the calculations of our previous example, we know that
\[
\|g_1\|^2=\sigma_1^2\frac{1-\alpha^2}{1+\alpha^2-2\alpha \cos \pi}=\sigma_1^2\frac{1-\alpha}{1+\alpha}.
\]
Now, we prove that $g_2 \in H^2$. We do so by first proving that $g_2$ is indeed in $L^2(\T)$, where $\T = \partial \D$. First, we note that
\[
g_2 =\sigma_2\frac{\exp\left(\frac{z+1}{z-1}\right)-1}{z+1},
\]
is analytic on $\D$ and its radial limits exist a.e. on $\T$. A straightforward calculation gives
\[
\left|\frac{\exp\left(\frac{e^{i\theta}+1}{e^{i\theta}-1}\right)-1}{e^{i\theta}+1}\right|^2=\frac{1-\cos\left(\frac{\sin \theta}{1-\cos \theta}\right)}{1+\cos \theta} = \frac{\sin ^2\left(\frac{1}{2}\cot \frac{\theta}{2}\right)}{\cos ^2\frac{\theta}{2}}.
\]
We put $x= \frac{1}{2} \cot \frac{\theta}{2}$, and change the variable from $\theta$ to $x$ to see
\[
\int_0^{2\pi}\left|\frac{\exp\left(\frac{e^{i\theta}+1}{e^{i\theta}-1}\right)-1}{e^{i\theta}+1}\right|^2d\theta =\int_0^{2\pi} \frac{\sin ^2\left(\frac{1}{2}\cot \frac{\theta}{2}\right)}{\cos ^2\frac{\theta}{2}} d\theta=\int_{-\infty}^\infty\frac{\sin ^2 x}{x^2}dx=\pi < \infty.
\]
Hence, $g_2 \in L^2(\T)$. Consider the Fourier series expansion of $g_2$ on $\T$ as
\[
g_2 =\sum_{n=-\infty}^{\infty} \alpha_n z^n,
\]
where $\alpha_n$, $n \in \Z$, are the Fourier coefficients. Now using the fact that
\[
(z + 1) g_2 = \sigma_2 (\vp_2-1 )\in H^2,
\]
we have, for any $n\geq 1$, that
\[
\langle (z + 1) g_2, \bar{z}^n\rangle = \langle \sigma_2(\vp_2-1 ), \bar{z}^n \rangle =0,
\]
and consequently
\[
\alpha_{-(n+1)} = - \alpha_{-n}.
\]
In particular, if $\alpha_{-m} \neq 0$, for some $m\geq 1$, then
\[
\alpha_{-n} = (-1)^{n-m} \alpha_{-m},
\]
for all $n\geq m$. This implies that the series $\sum_{n=m}^{\infty}|\alpha_{-n}|^2$ diverges, contradicting the fact that $g_2 \in L^2(\T)$. Therefore, we have
\[
\alpha_{-n} = 0,
\]
for all $n\geq1$, and hence $g_2 \in H^2$. Now, we compute the norm of $g_2$ as
\[
\|g_2\|^2=\sigma_2^2\left\|\frac{\exp\left(\frac{z+1}{z-1}\right)-1}{z+1}\right\|^2 = \frac{\sigma_2^2}{2 \pi} \int_0^{2\pi}\left|\frac{\exp\left(\frac{e^{i\theta}+1}{e^{i\theta}-1}\right)-1}{e^{i\theta}+1}\right|^2d\theta =\frac{\sigma_2^2}{2}.
\]
Hence, the relation $\sigma_2^2(1+\|g_1\|^2)=\sigma_1^2(1+\|g_2\|^2)$ is satisfied under the condition
\[
\frac{1}{\sigma_1^2}-\frac{1}{\sigma_2^2}=\frac{3\alpha-1}{2(1+\alpha)}.
\]
Therefore, in this case also, there is an abundance of examples of invariant subspaces of Brownian shifts, both unitarily equivalent and non-equivalent.
\end{example}

The argument used to prove that $g_2$ belongs to $H^2$ in the above proof is perhaps a standard method. In a more general context, this conclusion follows from \cite[Corollary 4.28]{Fricain}, which is a more involved result.

We will close this paper with the following result, proving that $\br$ is irreducible. Recall that a bounded linear operator $A$ acting on a Hilbert space $\mathcal{H}$ is irreducible if there is no nontrivial closed subspace of $\clh$ that reduces $A$.

\begin{proposition}
$\br$ on $H^2 \oplus \C$ is irreducible for all angle $\theta \in [0, 2\pi)$ and covariance $\sigma > 0$.
\end{proposition}
\begin{proof}
Fix $\theta \in [0, 2\pi)$ and $\sigma > 0$. Let $\clm$ be a nonzero closed subspace of $H^2 \oplus \C$ that reduces $\br$. If possible, assume that $\clm \subseteq H^2 \oplus \{0\}$. There exists a closed subspace $\tilde{\clm} \subseteq H^2$ such that $\clm = \tilde{\clm} \oplus \{0\}$. Then $\tilde{\clm}$ reduces $S$, and the irreducibility of $S$ implies that $\tilde{\clm} = H^2$. However, considering the adjoint of $\br$, it is evident that $\clm = H^2 \oplus \{0\}$ cannot reduce $\br$. Therefore, we must assume that $\clm \nsubseteq H^2 \oplus \{0\}$. There exists $\begin{bmatrix}f \\ \alpha \end{bmatrix} \in \clm$ such that $\alpha \neq 0$. Now, $\clm$ being a reducing subspace of $\br$, we have
\[
\begin{bmatrix}
f\\\sigma^2 \alpha + \alpha
\end{bmatrix} = \brst\br \begin{bmatrix}
f\\ \alpha \end{bmatrix} \in \clm.
\]
Then
\[
\begin{bmatrix} f\\ \sigma^2 \alpha + \alpha
\end{bmatrix} -  \begin{bmatrix}
f\\
\alpha
\end{bmatrix} =  \begin{bmatrix}
0\\ \sigma^2 \alpha
\end{bmatrix} \in \clm,
\]
implies $\begin{bmatrix} 0\\ 1\end{bmatrix}\in \clm$ (as $\sigma^2 \alpha \neq 0$). Since $\br \begin{bmatrix} 0\\ 1\end{bmatrix} =  \begin{bmatrix}\sigma\\e^{i\theta}\end{bmatrix} \in \clm$, it follows that
\[
\begin{bmatrix}\sigma\\ e^{i\theta}\end{bmatrix} - e^{i\theta}\begin{bmatrix} 0\\ 1\end{bmatrix} = \sigma  \begin{bmatrix} 1 \\ 0\end{bmatrix} \in \clm,
\]
and hence, $\begin{bmatrix} 1\\ 0\end{bmatrix} \in \clm$. Therefore, for any $n\geq 0$ we have $$ \br^n\begin{bmatrix}
1\\
0
\end{bmatrix} =\begin{bmatrix}
z^n\\
0
\end{bmatrix}  \in \clm. $$ Consequently, $\clm = H^2 \oplus \C$ and this proves the theorem.
\end{proof}

In closing, we remark that unitary equivalence or nonequivalence of operators arising from natural operators defined on Hilbert spaces is a fundamental and decades-old problem (cf. \cite{Ron, DS}). On one hand, it raises the question of defining new classes of operators from invariant subspaces, and on the other, it analyzes the characteristics of these operators under consideration. For instance, as already pointed out, among known operators, the shift on the Hardy space always yields unitarily equivalent invariant subspaces. At the other extreme, the Bergman shift and the Dirichlet shift never yield unitarily equivalent invariant subspaces \cite{Richter}. We have now enlarged this list by observing that the Brownian shift sometimes yields unitarily equivalent invariant subspaces and sometimes does not. This is particularly intriguing, as we have pointed out in \eqref{eqn: pert} that a Brownian shift is a rank-one perturbation of an isometry.

\vspace{0.2in}

\noindent\textbf{Acknowledgement:} We sincerely thank Professor Jan Stochel for his insightful discussions on the topic of this paper. We also thank the referee for carefully reading the manuscript and for the encouraging remarks. The first named author is supported by a post-doctoral fellowship provided by the National Board for Higher Mathematics (NBHM), India (order No.: 0204/10/(18)/2023/R\&D-II/2791 dated 28 February, 2023). The research of the second named author is supported by a post-doctoral fellowship provided by the National Board for Higher Mathematics (NBHM), India (Order No: 0204/16(8)/2024/R\&D- II/6760, dated May 09, 2024). The research of the third named author is supported in part by TARE (TAR/2022/000063) by SERB, Department of Science \& Technology (DST), Government of India.

\end{document}